\documentclass[10pt]{amsart}
\usepackage{amsmath,amssymb,latexsym,cite}
\usepackage[small]{caption}
\usepackage{graphicx,color,mathrsfs}
\usepackage{subfigure,color,wasysym}
\usepackage{cite}
\usepackage[colorlinks=true,urlcolor=blue,
citecolor=red,linkcolor=blue,linktocpage,pdfpagelabels,
bookmarksnumbered,bookmarksopen]{hyperref}
\usepackage[italian,english]{babel}
\usepackage{enumitem}
\usepackage[left=2.7cm,right=2.7cm,top=2.9cm,bottom=2.9cm]{geometry}
\usepackage[hyperpageref]{backref}

\numberwithin{equation}{section}
\newtheorem{theorem}{Theorem}[section]

\newtheorem{lemma}[theorem]{Lemma}

\theoremstyle{definition}

\renewcommand{\epsilon}{\eps}
\renewcommand{\i}{{\rm i}}

\newcommand{\N}{{\mathbb N}}
\newcommand{\R}{{\mathbb R}}

\renewcommand{\S}{{\mathbb S}}

\newcommand{\eps}{\varepsilon}

\newcommand{\pnorm}[2][]{\if #1'' \left|#2\right|_p \else \left|#2\right|_{#1} \fi}

\newcommand{\C}{\mathbb{C}}

\renewcommand{\O}{{\mathcal O}}

\renewcommand{\theta}{\vartheta}

\title[A magnetic Bourgain-Brezis-Mironescu formula]{Bourgain-Brezis-Mironescu formula \\ for magnetic operators}

\author[M.\ Squassina]{Marco Squassina}

\address[Marco Squassina]{Dipartimento di Informatica \newline\indent
Universit\`a degli Studi di Verona
\newline\indent
Strada Le Grazie 15, I-37134 Verona, Italy}
\email{\href{mailto:marco.squassina@univr.it}{marco.squassina@univr.it}}

\author[B.\ Volzone]{Bruno Volzone}

\address[Bruno Volzone]{Dipartimento di Ingegneria  \newline\indent
	Universit\`a di Napoli Parthenope
	\newline\indent
	Centro Direzionale Isola C/4 80143 Napoli, Italy}
\email{\href{mailto:bruno.volzone@uniparthenope.it}{bruno.volzone@uniparthenope.it}}



\begin{document}

\begin{abstract}
We prove a Bourgain-Brezis-Mironescu type formula 
for a class of nonlocal magnetic spaces, which builds a bridge
between a fractional magnetic operator recently introduced and the 
classical theory.
\end{abstract}

\maketitle

\section{Introduction}
\noindent
Let $s\in (0,1)$ and $N>2s$.\ If $A:\R^N\to\R^N$ is a smooth function, the nonlocal operator
\begin{equation*}
(-\Delta)^s_Au(x)=c(N,s) \lim_{\eps\searrow 0}\int_{B^c_\eps(x)}\frac{u(x)-e^{\i (x-y)\cdot A\left(\frac{x+y}{2}\right)}u(y)}{|x-y|^{N+2s}}dy,
\,\,\quad x\in\R^N,
\end{equation*}
has been recently introduced in \cite{piemar}, where the ground state solutions of $(-\Delta)^s_A u+u=|u|^{p-2}u$ 
in the three dimensional setting have been obtained via concentration compactness arguments. 
If $A=0$, then the above operator is consistent with the usual notion of fractional Laplacian.  
The motivations that led to its introduction are carefully described in \cite{piemar} and rely essentially
on the L\'evy-Khintchine formula for the generator of a general L\'evy process. We point out that the 
normalization constant $c(N,s)$ satisfies
$$
\lim_{s\nearrow 1}\frac{c(N,s)}{1-s}=\frac{4N\Gamma(N/2)}{2\pi^{N/2}},
$$
where $\Gamma$ denotes the Gamma function.\ For the sake of completeness, we recall that 
different definitions of nonlocal magnetic operator are viable, see e.g.\ \cite{I10,franketal}. All these notions
aim to extend the well-know definition of the magnetic Schr\"odinger operator 
$$
-\big(\nabla-\i A(x)\big)^2u=-\Delta u+2\i A(x)\cdot \nabla u+|A(x)|^2u+\i u\, {\rm div} A(x),
$$
namely the differential of the energy functional
$$
{\mathcal E}_A(u)=\int_{\R^N}|\nabla u-\i A(x)u|^2dx,
$$
for which we refer the reader to \cite{arioliSz,AHS,reed} and the included references. In order to corroborate the justification for the
introduction of $(-\Delta)^s_A$, in this note we prove that a well-known formula due to 
Bourgain, Brezis and Mironescu (see \cite{bourg,bourg2,mazia}) for the limit of the Gagliardo semi-norm of $H^s(\Omega)$ as $s\nearrow 1$ extends to 
the magnetic setting. As a consequence, in a suitable sense, from the nonlocal to the local regime, it holds 
$$
(-\Delta)^s_Au   \leadsto \big(\nabla-\i A(x)\big)^2u,\quad\,\,\, \text{for $s\nearrow 1$.}
$$
We consider
$$
[u]_{H^1_A(\Omega)}:=\sqrt{\int_{\Omega}|\nabla u-\i A(x)u|^2dx},
$$
and define $H^1_A(\Omega)$  as the space of functions $u\in L^2(\Omega,\C)$ such that  $[u]_{H^1_A(\Omega)}<\infty$ endowed with the norm
$$
\|u\|_{H^1_A(\Omega)}:=\sqrt{\|u\|_{L^2(\Omega)}^2+[u]_{H^1_A(\Omega)}^2}.
$$
Our main results are the following

\begin{theorem}[Magnetic Bourgain-Brezis-Mironescu]
	\label{main}
	Let $\Omega\subset\R^N$ be an open bounded set with Lipschitz boundary and $A\in C^2(\bar{\Omega})$.\ Then,
	for every  $u\in H^1_{A}(\Omega)$, we have
	$$
	\lim_{s\nearrow 1}(1-s)\int_{\Omega}\int_{\Omega}\frac{|u(x)-e^{\i (x-y)\cdot A\left(\frac{x+y}{2}\right)}u(y)|^2}{|x-y|^{N+2s}}dxdy=
	K_N\int_{\Omega}|\nabla u-\i A(x)u|^2dx,
	$$
	where 
	\begin{equation}
	\label{valoreK}
	K_{N}=\frac{1}{2}\int_{\S^{N-1}}|\omega\cdot {\bf e}|^{2}d\mathcal{H}^{N-1}(\omega),
	\end{equation}
	being $\S^{N-1}$ the unit sphere and ${\bf e}$ any unit vector in $\R^{N}$. 
\end{theorem}

\noindent
As a variant of Theorem~\ref{main}, if $H^1_{0,A}(\Omega)$ denotes the closure of $C^\infty_c(\Omega)$ in $H^1_A(\Omega)$, we get the following

\begin{theorem}
	\label{main2}
	Let $\Omega\subset\R^N$ be an open bounded set with Lipschitz boundary.
	Assume that $A:\R^N\to\R^N$ is locally bounded and $A\in C^2(\bar{\Omega})$.\ Then,
	for every  $u\in H^1_{0,A}(\Omega)$, we have
	$$
	\lim_{s\nearrow 1}(1-s)\int_{\R^{2N}}\frac{|u(x)-e^{\i (x-y)\cdot A\left(\frac{x+y}{2}\right)}u(y)|^2}{|x-y|^{N+2s}}dxdy=
	K_N\int_{\Omega}|\nabla u-\i A(x)u|^2dx.
	$$
\end{theorem}


\vskip14pt
\noindent
{\bf Notations.} Let $\Omega\subset\R^N$ be an open set.\ We denote by $L^2(\Omega,\C)$ 
the Lebesgue space of complex valued functions with summable square. 
For $s\in (0,1)$, the magnetic Gagliardo semi-norm is 
$$
[u]_{H^s_A(\Omega)}:=\sqrt{\int_{\Omega}\int_{\Omega}\frac{|u(x)-e^{\i (x-y)\cdot A\left(\frac{x+y}{2}\right)}u(y)|^2}{|x-y|^{N+2s}}dxdy}.
$$
We denote by $H^s_A(\Omega)$ the space of functions $u\in L^2(\Omega,\C)$ such that  $[u]_{H^s_A(\Omega)}<\infty$ endowed with 
$$
\|u\|_{H^s_A(\Omega)}:=\sqrt{\|u\|_{L^2(\Omega)}^2+[u]_{H^s_A(\Omega)}^2}.
$$
We denote by $B(x_0,R)$ the ball in $\R^N$ of center $x_0$ and radius $R>0$.\
For any set $E\subset \R^N$ we will denote by $E^c$ the complement
of $E$. For $A,B\subset\R^N$ open and bounded, $A\Subset B$ means $\bar A\subset B$.
\smallskip

\section{Preliminary results}

\noindent
We start with the following Lemma.

\begin{lemma}
	\label{stima1}
Assume that $A:\R^N\to\R^N$ is locally bounded. Then, 
for any	compact $V\subset \R^N$ with $\Omega \Subset V$,
there exists $C=C(A,V)>0$ such that
\begin{equation*}
\int_{\R^{N}} |u(y+h)-e^{\i  h\cdot A\left(y+\frac{h}{2}\right)}u(y)|^{2}dy\leq C |h|^{2} \|u\|^{2}_{H^1_{A}(\R^{N})}, 
\end{equation*}
for all $u\in H^1_{A}(\R^{N})$ such that $u=0$ on $V^c$ and any $h\in \R^{N}$ with $|h|\leq 1$.
\end{lemma}
\begin{proof}
Assume first that $u\in C_{0}^{\infty}(\R^{N})$ 
with $u=0$ on $V^c$.  Fix $y, h\in \R^{N}$ and define
\[
\varphi(t):=e^{\i (1-t) h\cdot A\left(y+\frac{h}{2}\right)}u(y+th),\quad\, t\in [0,1].
\]
Then we have
\[
u(y+h)-e^{\i  h\cdot A\left(y+\frac{h}{2}\right)}u(y)=\varphi(1)-\varphi(0)=\int_{0}^{1}\varphi^{\prime}(t)dt,
\]
and since
\[
\varphi^{\prime}(t)=e^{\i (1-t) h\cdot A\big(y+\frac{h}{2}\big)}\,h\cdot\Big(\nabla_{y}u(y+th)-\i A\Big(y+\frac{h}{2}\Big)u(y+th)\Big),
\]
by H\"older inequality we get
\[
|u(y+h)-e^{\i  h\cdot A\left(y+\frac{h}{2}\right)}u(y)|^{2}\leq |h|^{2}\int_{0}^{1} \Big|\nabla_{y}u(y+th)-\i A\Big(y+\frac{h}{2}\Big)u(y+th)\Big|^{2}dt.
\]
Therefore, integrating with respect to $y$ over $\R^N$ and using Fubini's Theorem, we get
\begin{align*}
\int_{\R^{N}} |u(y+h)-e^{\i  h\cdot A\left(y+\frac{h}{2}\right)}u(y)|^{2}dy &\leq |h|^{2}
\int_{0}^{1}dt\int_{\R^{N}} \Big|\nabla_{y}u(y+th)-\i A\Big(y+\frac{h}{2}\Big)u(y+th)\Big|^{2}dy\nonumber\\
&=|h|^{2} \int_{0}^{1}dt\int_{\R^{N}} \Big|\nabla_{z}u(z)-\i A\Big(z+\frac{1-2t}{2}h\Big)u(z)\Big|^{2}dz\nonumber\\
&\leq 2 |h|^{2} \int_{\R^{N}} |\nabla_{z}u(z)-\i A\left(z\right)u(z)|^{2}dz \\&
+2|h|^{2} \int_{V} \Big|A\Big(z+\frac{1-2t}{2}h\Big)-A(z)\Big|^{2}|u(z)|^{2}dz.
\end{align*}
Then, since $A$ is bounded on the set $V$, we have for some constant $C>0$
\begin{align*}
\int_{\R^{N}} |u(y+h)-e^{\i  h\cdot A\left(y+\frac{h}{2}\right)}u(y)|^{2}dy&\leq C |h|^{2}\left( \int_{\R^{N}} |\nabla_{z}u(z)-\i
A\left(z\right)u(z)|^{2}dz+ \int_{\R^{N}} |u(z)|^{2}dz\right)\nonumber\\&=C |h|^{2} \|u\|^{2}_{H^1_{A}(\R^{N})}.
\end{align*}
When dealing with a general $u$ we can argue by a density argument.
\end{proof}

\begin{lemma}
	\label{extlem}
Let $\Omega\subset\R^N$ be an open bounded set with Lipschitz boundary,
$V\subset\R^N$ a compact set with $\Omega \Subset V$ and $A:\R^N\to\R^N$ locally bounded.
Then there exists $C(\Omega,V,A)>0$ such that for any $u\in H^1_A(\Omega)$
there exists $Eu\in H^1_A(\R^N)$ such that $Eu=u$ in $\Omega$, $Eu=0$ in $V^c$ and
$$
\|Eu\|_{H^1_A(\R^N)}\leq C(\Omega,V,A)\|u\|_{H^1_A(\Omega)}.
$$
\end{lemma}
\begin{proof}
Observe that, for any bounded set $W\subset\R^N$ there exist $C_1(A,W),C_2(A,W)>0$ with
$$
C_1(A,W)\|u\|_{H^1(W)}\leq \|u\|_{H^1_A(W)}\leq C_2(A,W)\|u\|_{H^1(W)}, 
\quad\text{for any $u\in H^1(W)$.}
$$
This follows easily, via simple computations, by the definition of the norm of $H^1_A(W)$ and 
in view of the local boundedness assumption on the potential $A$.
Now, by the standard extension property for $H^1(\Omega)$ (see e.g.\ \cite[Theorem 1, p.254]{evans})
there exists $C(\Omega,V)>0$ such that for any $u\in H^1(\Omega)$
there exists a function $Eu\in H^1(\R^N)$ such that $Eu=u$ in $\Omega$, $Eu=0$ in $V^c$ and
$\|Eu\|_{H^1(\R^N)}\leq C(\Omega,V)\|u\|_{H^1(\Omega)}.$ Then, for any $u\in H^1_A(\Omega)$, we get
\begin{align*}
\|Eu\|_{H^1_A(\R^N)}&=\|Eu\|_{H^1_A(V)}\leq C_2(A,V) \|Eu\|_{H^1(V)}= C_2(A,V) \|Eu\|_{H^1(\R^N)} \\
&\leq C(\Omega,V)C_2(A,V)\|u\|_{H^1(\Omega)}\leq C(\Omega,V)C_2(A,V)C_1^{-1}(A,\Omega)\|u\|_{H^1_A(\Omega)},
\end{align*}
which concludes the proof.
%
%
%
\end{proof}

\noindent
We can now prove the following result:

\begin{lemma}\label{firstLemma}
Let $A:\R^N\to\R^N$ be locally bounded.
Let $u\in H^1_{A}(\Omega)$ and $\rho\in L^{1}(\R^{N})$ with $\rho\geq0$. Then
\[
\int_{\Omega}\int_{\Omega}\frac{|u(x)-e^{\i (x-y)\cdot A\left(\frac{x+y}{2}\right)}u(y)|^2}{|x-y|^{2}}\rho(x-y)\,dxdy\leq C\|\rho\|_{L^{1}}\|u\|^{2}_{H^1_{A}(\Omega)}
\]
where $C$ depends only on $\Omega$ and $A$.
\end{lemma}
\begin{proof}
Let $V\subset \R^N$ be a fixed compact set with $\Omega\Subset V$. 
Given $u\in H^1_A(\Omega)$, by Lemma~\ref{extlem},
there exists a function $\tilde u\in H^{1}_{A}(\R^{N})$ with $\tilde u=u$ on $\Omega$ and $\tilde u=0$ on $V^c$.
By Lemma \ref{stima1} and \ref{extlem}, 
\begin{equation}
\int_{\R^{N}} |\tilde u(y+h)-e^{\i  h\cdot A\left(y+\frac{h}{2}\right)}\tilde u(y)|^{2}dy
\leq C |h|^{2} \|\tilde u\|^{2}_{H^1_{A}(\R^{N})}\leq C |h|^{2} \|u\|^{2}_{H^1_{A}(\Omega)}, \label{secondeq}
\end{equation}
for some positive constant $C$ depending on $\Omega$ and $A$. Then, in light of \eqref{secondeq}, we get
\begin{align*}
\int_{\Omega}\int_{\Omega}\frac{|u(x)-e^{\i (x-y)\cdot A\left(\frac{x+y}{2}\right)}u(y)|^2}{|x-y|^{2}}\rho(x-y)\,dxdy
&\leq\int_{\R^{N}}\int_{\R^{N}}\rho(h)\frac{|\tilde u(y+h)-e^{\i  h\cdot A\left(y+\frac{h}{2}\right)}\tilde u(y)|^{2}}{|h|^{2}}dydh\nonumber
\\& =\int_{\R^{N}}\frac{\rho(h)}{|h|^{2}}\Big(\int_{\R^{N}}|\tilde u(y+h)-e^{\i  h\cdot A\left(y+\frac{h}{2}\right)}\tilde u(y)|^{2}dy\Big)dh
\\& \leq C\|\rho\|_{L^{1}} \|u\|^{2}_{H^1_{A}(\Omega)},
\end{align*}
which concludes the proof.
\end{proof}


\begin{lemma}\label{lemmino}
	Let $A:\R^N\to\R^N$ be locally bounded
	and let $u\in H^1_{0,A}(\Omega)$.\ Then, we have 
	\[
	(1-s)\int_{\R^{2N}}\frac{|u(x)-e^{\i (x-y)\cdot A\left(\frac{x+y}{2}\right)}u(y)|^2}{|x-y|^{N+2s}}dxdy\leq C\|u\|^{2}_{H^1_{A}(\Omega)}
	\]
	where $C$ depends only on $\Omega$ and $A$.
\end{lemma}
\begin{proof}
	Given $u\in C^\infty_c(\Omega)$, by  Lemma \ref{stima1} we have
	\begin{equation*}
	\int_{\R^{N}} |u(y+h)-e^{\i  h\cdot A\left(y+\frac{h}{2}\right)}u(y)|^{2}dy\leq C |h|^{2} \|u\|^{2}_{H^1_{A}(\Omega)}, 
	\end{equation*}
	for some $C>0$ depending on $\Omega$ and $A$ and all $h\in\R^N$ with $|h|\leq 1$.
	Then, we get
	\begin{align*}
	&(1-s)\int_{\R^{2N}}\frac{|u(x)-e^{\i (x-y)\cdot A\left(\frac{x+y}{2}\right)}u(y)|^2}{|x-y|^{N+2s}}\,dxdy
	\leq (1-s)\int_{\R^{2N}}\frac{|u(y+h)-e^{\i  h\cdot A\left(y+\frac{h}{2}\right)}u(y)|^{2}}{|h|^{N+2s}}dydh   \\
	& =(1-s)\int_{\{|h|\leq 1\}}\frac{1}{|h|^{N+2s}}\Big(\int_{\R^{N}}|u(y+h)-e^{\i  h\cdot A\left(y+\frac{h}{2}\right)}u(y)|^{2}dy\Big)dh   \\
	& +4(1-s)\int_{\{|h|\geq 1\}}\frac{1}{|h|^{N+2s}}dh \|u\|_{L^2(\Omega)}^2   \\
	& \leq (1-s)\int_{\{|h|\leq 1\}}\frac{1}{|h|^{N+2s-2}}dh\|u\|^{2}_{H^1_{A}(\Omega)}+C\|u\|_{L^2}^2  \leq C\|u\|^{2}_{H^1_{A}(\Omega)}.
	\end{align*}
	The assertion then follows by a density argument.
\end{proof}


\noindent 
If $A|_{\Omega}$ is smooth (and extended if necessary to a locally bounded field on $\Omega^c$), we get the following result.

\begin{theorem}
	\label{general}
Assume that $A\in C^2(\bar{\Omega})$. Let $u\in H^1_{A}(\Omega)$ and consider a sequence 
$\left\{\rho_{n}\right\}_{n\in\N}$ of nonnegative radial functions in $L^1(\R^N)$ with
\begin{equation}
\label{normaliz}
\lim_{n\to\infty}\int_{0}^\infty\rho_{n}(r) r^{N-1}dx=1,
\end{equation}
and such that, for every $\delta>0$,
\begin{equation}
\lim_{n\to\infty}\int_{\delta}^{\infty}\rho_{n}(r)r^{N-1}dr=0
\label{fourtheq}
\end{equation}
Then, we have
\begin{equation}
\lim_{n\to\infty}\int_{\Omega}\int_{\Omega}\frac{|u(x)-e^{\i (x-y)\cdot A\left(\frac{x+y}{2}\right)}u(y)|^2}{|x-y|^{2}}\rho_{n}(x-y)\,dxdy=2K_{N} \int_{\Omega}|\nabla u-\i A(x) u|^{2}\,dx\label{thirdeq}
\end{equation}
being $K_{N}$ the constant introduced in \eqref{valoreK}.
\end{theorem}
\begin{proof}
Let us first observe that by \eqref{normaliz} and \eqref{fourtheq} we easily obtain that, for every $\delta>0$,
\begin{equation}
\lim_{n\to\infty}\int_{0}^{\delta}\rho_{n}(r)r^{N}dr=
\lim_{n\to\infty}\int_{0}^{\delta}\rho_{n}(r)r^{N+1}dr=0.
\label{fourtheq-2}
\end{equation}
In fact, taken any $0<\tau<\delta$, we have
$$
\int_{0}^{\delta}\rho_{n}(r)r^{N}dr=\int_{0}^{\tau}\rho_{n}(r)r^{N}dr+\int_{\tau}^{\delta}\rho_{n}(r)r^{N}dr\leq
\tau\int_{0}^{\tau}\rho_{n}(r)r^{N-1}dr+\delta\int_{\tau}^{\infty}\rho_{n}(r)r^{N-1}dr,
$$
from which formula \eqref{fourtheq-2} follows using \eqref{normaliz}, \eqref{fourtheq} and letting $\tau\searrow 0$.
We follow the main lines of the proof in \cite{bourg}. Setting
\[
F_{n}^{u}(x,y):=\frac{u(x)-e^{\i (x-y)\cdot A\left(\frac{x+y}{2}\right)}u(y)}{|x-y|}\rho_{n}^{1/2}(x-y),
\quad\,\,\, x,y\in\Omega,\,\, n\in\N,
\]
by virtue of Lemma \ref{firstLemma}, for all $u,v\in H^{1}_{A}(\Omega)$, recalling \eqref{normaliz} we have
\[
\big|\|F_{n}^{u}\|_{L^2(\Omega\times\Omega)}-\|F_{n}^{v}\|_{L^2(\Omega\times\Omega)}\big|\leq\|F_{n}^{u}-F_{n}^{v}\|_{L^2(\Omega\times\Omega)}
\leq C  \|u-v\|_{H^{1}_{A}(\Omega)},
\]
for some $C>0$ depending on $\Omega$ and $A$.
This allows to prove \eqref{thirdeq} for $u\in C^{2}(\bar{\Omega})$. If we set
\[
\varphi(y):=e^{\i (x-y)\cdot A\left(\frac{x+y}{2}\right)}u(y), 
\]
since
$$
\nabla_{y}\varphi(y)=e^{\i (x-y)\cdot A\left(\frac{x+y}{2}\right)}\Big(\nabla_{y}u(y)-\i A\Big(\frac{x+y}{2}\Big) u(y)+\frac{\i}{2}\, u(y) (x-y)\cdot
\nabla_y A \Big(\frac{x+y}{2}\Big)\Big),
$$
if $x\in\Omega,$ a second order Taylor expansion gives (since $u,A\in C^2$, then $\nabla^2_y \varphi$ is bounded on $\bar\Omega$)
$$
u(x)-e^{\i (x-y)\cdot A\left(\frac{x+y}{2}\right)}u(y)=
\varphi(x)-\varphi(y)=(\nabla u(x)-\i A(x) u(x))\cdot (x-y)+ \O(|x-y|^2).
$$
Hence, for any fixed $x\in\Omega$,
\begin{equation}
\frac{\big|u(x)-e^{\i (x-y)\cdot A\left(\frac{x+y}{2}\right)}u(y)\big|}{|x-y|}=\left|(\nabla u(x)-\i A(x) u(x))\cdot \frac{x-y}{|x-y|}\right|+\O(|x-y|).
\label{fiftheq}
\end{equation}
Fix $x\in\Omega$. If we set $R_x:=\text{dist}(x,\partial\Omega)$, integrating with respect to $y$, we have
\begin{align}
\int_{\Omega}\frac{|u(x)-e^{\i (x-y)\cdot A\left(\frac{x+y}{2}\right)}u(y)|^2}{|x-y|^{2}}\rho_{n}(x-y)\,dy
&=\int_{B(x,R_x)}\frac{|u(x)-e^{\i (x-y)\cdot A\left(\frac{x+y}{2}\right)}u(y)|^2}{|x-y|^{2}}\rho_{n}(x-y)\,dy\nonumber\\&
+\int_{\Omega\setminus B(x,R_x)}\frac{|u(x)-e^{\i (x-y)\cdot A\left(\frac{x+y}{2}\right)}u(y)|^2}{|x-y|^{2}}\rho_{n}(x-y)\,dy.
\label{sixtheq}
\end{align}
The second integral goes to zero by conditions \eqref{fourtheq}, since
$$
\lim_{n\to\infty}\int_{\Omega\setminus B(x,R_x)}\frac{|u(x)-e^{\i (x-y)\cdot A\left(\frac{x+y}{2}\right)}u(y)|^2}{|x-y|^{2}}\rho_{n}(x-y)\,dy\leq
C\lim_{n\to\infty}\int_{B^c(0,R_x)}\rho_n(z)dz=0.
$$
Now, in light of \eqref{fiftheq}, following \cite{bourg} we compute
\begin{align*}
\int_{B(x,R_x)}\frac{|u(x)-e^{\i (x-y)\cdot A\left(\frac{x+y}{2}\right)}u(y)|^2}{|x-y|^{2}}\rho_{n}(x-y)\,dy&
= Q_{N} |\nabla u(x)-\i A(x) u(x)|^{2}\int_{0}^{R_x}r^{N-1}\rho_{n}(r)dr\\&+\O\left(\int_{0}^{R_x} r^{N}\rho_{n}(r)dr\right)
+\O\left(\int_{0}^{R_x} r^{N+1}\rho_{n}(r)dr\right),
\end{align*}
where we have set
\[
Q_{N}=\int_{\S^{N-1}}|\omega\cdot {\bf e}|^{2}d\mathcal{H}^{N-1}(\omega),
\]
being ${\bf e}\in \R^N$ a unit vector.\ Letting $n\to\infty$ in \eqref{sixtheq}, 
the result follows by dominated convergence, taking into account
formulas \eqref{fourtheq-2}.
\end{proof}

\section{Proofs of Theorem~\ref{main} and \ref{main2}}

\subsection{Proof of Theorem~\ref{main}.}
If $r_\Omega:={\rm diam}(\Omega)$, we consider 
a radial cut-off $\psi\in C^\infty_c(\R^N)$, $\psi(x)=\psi_0(|x|)$ with $\psi_0(t)=1$ for $t<r_\Omega$ 
and $\psi_0(t)=0$ for $t>2r_\Omega$. Then, by construction, $\psi_0(|x-y|)=1,$ for every $x,y\in\Omega$.
Furthermore, let $\{s_n\}_{n\in\N}\subset (0,1)$ be a sequence with $s_n\nearrow 1$ as $n\to\infty$ and consider the sequence of 
radial functions in $L^1(\R^N)$
\begin{equation}
\label{def-rho}
\rho_n(|x|)=\frac{2(1-s_n)}{|x|^{N+2s_n-2}}\psi_0(|x|),
\,\,\quad x\in\R^N,\,\, n\in\N.
\end{equation}
Notice that \eqref{normaliz} holds, since
\begin{equation*}
\lim_{n\to\infty}\int_{0}^{r_\Omega}\rho_n(r)r^{N-1}dr
=\lim_{n\to\infty}2(1-s_n)\int_{0}^{r_\Omega}\frac{1}{r^{2s_n-1}}dr=\lim_{n\to\infty} r_\Omega^{2-2s_n}=1,
\end{equation*}
and
$$
\lim_{n\to\infty}\int_{r_\Omega}^{2 r_\Omega}\rho_n(r) r^{N-1}dr=\lim_{n\to\infty}2(1-s_n)\int_{r_\Omega}^{2r_\Omega}\frac{\psi_0(r)}{t^{2s_n-1}}dr=0.
$$
In a similar fashion, for any $\delta>0$, there holds
\begin{equation*}
\lim_{n\to\infty}\int_{\delta}^{\infty}\rho_{n}(r)r^{N-1}dr\leq \lim_{n\to\infty} 2(1-s_n)\int_{\delta}^{2r_\Omega}\frac{1}{t^{2s_n-1}}dt=0.
\end{equation*}
Then Theorem~\ref{main} follows directly from Theorem~\ref{general} using $\rho_n$ as defined in \eqref{def-rho}.  \qed

\subsection{Proof of Theorem~\ref{main2}.} In light of Theorem~\ref{main} and since $u=0$ on $\Omega^c$, we have
$$
\lim_{s\nearrow 1}(1-s)\int_{\R^{2N}}\frac{|u(x)-e^{\i (x-y)\cdot A\left(\frac{x+y}{2}\right)}u(y)|^2}{|x-y|^{N+2s}}dxdy=
	K_N\int_{\Omega}|\nabla u-\i A(x)u|^2dx+\lim_{s\nearrow 1}R_s,
$$	
where 
$$
R_s\leq 2(1-s)\int_{\Omega}\int_{\R^{N}\setminus\Omega}\frac{|u(x)|^2}{|x-y|^{N+2s}}dxdy.
$$
On the other hand, arguing as in the proof of \cite[Proposition 2.8]{braparsqu}, we get $R_s\to 0$
as $s\nearrow 1$ when $u\in C^\infty_c(\Omega)$ and, on account of Lemma~\ref{lemmino}, 
for general function in $H^1_{0,A}(\Omega)$ by a density argument.  \qed

\vskip12pt
\noindent
{\bf Acknowledgments.}\
The authors kindly thank the anonymous Referee for his/her very careful reading of the manuscript
and for providing useful remarks helpful to improve the content. \\
The authors are members of Gruppo Nazionale per l'Analisi Matematica, la Probabilit\`a
e le loro Applicazioni (GNAMPA).

\medskip

\end{document}